\documentclass[11pt]{amsproc}
\usepackage{mathrsfs}
\usepackage{stmaryrd}
\usepackage{cases}
\usepackage{amssymb}
\usepackage{amsmath}
\usepackage{amsfonts}
\usepackage{graphicx}
\usepackage{amsmath,amstext,amsbsy,amssymb, color}
\newtheorem{theorem}{Theorem}[section]
\newtheorem{lemma}[theorem]{Lemma}
\newtheorem{proposition}[theorem]{Proposition}

\theoremstyle{definition}
\newtheorem{definition}[theorem]{Definition}

\theoremstyle{remark}
\newtheorem{remark}[theorem]{Remark}

\numberwithin{equation}{section} \errorcontextlines=0

\newcommand{\Pf}{\mathrm{Pf}}

\newcommand{\ot}{\otimes}

\newcommand{\sign}{\mathrm{sign}}

\begin{document}

\dedicatory{Dedicated to Vyjayanthi Chari in honor of her 60th birthday}

\title[Dynamical quantum determinants and Pfaffians]
{Dynamical quantum determinants and Pfaffians}
\author{Naihuan Jing}
\address{NJ: Department of Mathematics, North Carolina State University, Raleigh, NC 27695, USA}
\email{jing@math.ncsu.edu}
\author{Jian Zhang}
\address{JZ: Institute of Mathematics and Statistics, University of Sao Paulo, Sao Paulo, Brazil 05315-970}
\email{zhang@ime.usp.br}
\thanks{{\scriptsize
\hskip -0.6 true cm MSC (2010): Primary: 17B37; Secondary: 58A17, 15A75, 15B33, 15A15.
}}

\begin{abstract}
We introduce the dynamical quantum Pfaffian on the dynamical quantum general linear group and prove its
fundamental transformation identity. Hyper quantum dynamical Pfaffian is also introduced and formulas connecting them
are given.
\end{abstract}
\maketitle
\section{Introduction}

Dynamical quantum groups are important generalization of quantum groups
introduced by Etingof and Varchenko \cite{EV} in connection with the elliptic quantum groups \cite{F, FV}.
See the review \cite{ES} for the background and related literature as well as comparison with
usual quantum groups (see \cite{CP}).
The dynamical quantum group is in fact some quantum groupoid, thus also related to the deformation of the Poisson
groupoid \cite{L, T, X}.  In this paper we essentially follow \cite{EV} to define the dynamical quantum group with some modification \cite{KN}.

As discussed in \cite{FRT} in general, given an R-matrix one can associate certain quantum semigroup $A(R)$ via the $RTT$ formulation.
Let $V$ be the complex $n$-dimensional vector space with basis $v_i$ and dual basis $\lambda_i$ for $V^*$.
We consider the dynamical R-matrix $R(\underline{\lambda})$ defined on $V\otimes V$ as follows.
\begin{align}\notag
R(\underline{\lambda})&=q\sum_{i=1}^ne_{ii}\otimes e_{ii}+\sum_{i<j}^ne_{ii}\otimes e_{jj}+
\sum_{i>j}^ng(\lambda_i-\lambda_j)e_{ii}\otimes e_{jj}\\
&\qquad\qquad\qquad +\sum_{i\neq j}^nh(\lambda_i-\lambda_j)e_{ii}\otimes e_{jj}
\end{align}
where $\underline{\lambda} = (\lambda_1, \ldots, \lambda_n)$, $e_{ij}$ are the unit matrix elements inside $\mathrm{End}(V)$ such that
$e_{ij}v_k=\delta_{jk}v_i$, and $g, h$ are
certain $q$-analog functions on $V$ (see \eqref{h0g}).
The associated bialgebra $\mathscr{F}_R(M(n))$ is called the dynamical quantum group in the general linear type, which generalizes
the usual quantum general linear semigroup $M_q(n)$. The  non-dynamical quantum semigroup $M_q(n)$ becomes the quantum group $\mathrm{GL}_q(n)$ with the help of
a special central element called the quantum determinant $\det_q$.

On the quantum semigroup $M_q(n)$ one can develop a theory of
quantum linear algebra
\cite{HH}, and introduce quantum determinants and minors. One can prove
key equations such as the quantum Cramer identity, Cayley-Hamilton identity etc \cite{TT, Ma}.
For a unified treatment using Manin's quadratic algebras, see
\cite{JZ1,JZ3}.

Correspondingly on the dynamical quantum general linear semigroupoid $\mathscr{F}_R(M(n))$, we can also
introduce the dynamical quantum determinant, minors and prove that they also enjoy similar favorable properties
\cite{HIO, KN, KN2}. It turns out that the quantum dynamical determinant is also a central group-like element, and the
Laplace expansions for quantum dynamical minors also are satisfied in a manner similar to the non-dynamical quantum situation.
In particular the quantum dynamical determinant also turns $\mathscr{F}_R(M(n))$ into a dynamical quantum groupoid \cite{S, vN}.

The gaol of this paper is to introduce the quantum dynamical Pfaffian and show that
it enjoys favorable properties similar to the quantum group situation \cite{JZ}. 
Our main technique is to use quadratic algebras or quantum de Rham complexes \cite{Ma} to study quantum
determinants and quantum Pfaffians, and express them as the scaling constants of
quantum differential forms (cf. \cite{JZ1}).
In particular, we prove that the dynamical quantum Pfaffian satisfies the transformation property:
\begin{equation}
\Pf(ABA^t)=\det(A)\Pf(B)
\end{equation}
even though the identity $\Pf(A)=\sqrt{\det(A)}$ no longer holds for the quantum anti-symmetric matrix.

The paper is organized as follows. In section two, we introduce the dynamical quantum general linear group via the
generalized quantum Yang-Baxter R-matrix and review the basic information on quantum dynamical minors
and determinants. In section three, we give a factorization formula for the dynamical quantum determinant in terms of quasi-determinant of Gelfand and Retakh. In section four, we study the dynamical quantum Pfaffians using $q$-forms. In the last section, quantum dynamical hyper-Pfaffians are
given and their fundamental properties and identities are discussed.

\section{Dynamical analogue of the quantum algebra M(n).}
In this section, we recall some basic facts about dynamical quantum groups \cite{KN}.

Let $\mathfrak{h}^*$ be the dual space of the $n$-dimensional commutative Lie algebra $\mathfrak h$ and we fix
a linear basis $\{ e_i\}$ of $\mathfrak{h}^*$, so $\mathfrak{h}^*$ can be identified with $\mathbb C^n$.
For $[1,n]=\{1,2,\ldots,n\}$, define $\omega:[1,n]\rightarrow \mathfrak{h}^* $ by $\omega(i)=e_i$.

Fix a generic $q\in\mathbb C^{\times}$. For $\lambda\in \mathfrak h^*$, the functional $q^{\lambda}: \mathfrak h\longrightarrow \mathbb C$ is defined as usual by
$v\mapsto q^{\lambda(v)}, v\in \mathfrak h$. We denote by $h(\lambda)$ and $g(\lambda)$ the following
special meromorphic functionals on $\mathfrak h$: 
\begin{equation}\label{h0g}
\begin{split}
h(\lambda)&=q\frac{q^{-2\lambda}-q^{-2}}{q^{-2\lambda}-1},\\
g(\lambda)&=h(\lambda)h(-\lambda)=\frac{(q^{-2\lambda}-q^{-2})(q^{-2\lambda}-q^2)}{(q^{-2\lambda}-1)^2}.
\end{split}
\end{equation}

Let $M_{\mathfrak h^*}$ be the space of meromorphic functionals on $\mathfrak h^*$. In particular, the above $f(\lambda)$, $g(\lambda)$ are
elements inside $M_{\mathfrak h^*}$.
Let $\mathfrak{h}$-algebra $\mathscr{F}_R(M(n))$ be the associative algebra generated by the elements $t_{ij}$,
$1\leq i,j \leq n$ together with two copies of $M_{\mathfrak{h}^*}$. The elements of the two copies $M_{\mathfrak{h}^*}$ are
$f(\underline{\lambda})= f(\lambda_1,\ldots,\lambda_n)$ and $f(\underline{\mu})= f(\mu_1,\ldots,\mu_n)$,
embedded as subalgebras. Here $\lambda_i$ (resp. $\mu_i$) is a function on
$\mathfrak h$. The defining relations of $\mathscr{F}_R(M(n))$ 
consist of two types. The first group of relations are given by
\begin{equation}
\begin{split}
f_1(\underline{\lambda}) f_2(\underline{\mu})=f_2(\underline{\mu})f_1(\underline{\lambda}),\\
f(\underline{\lambda})t_{ij}=t_{ij}f(\underline{\lambda}+\omega(i)),\\
f(\underline{\mu})t_{ij}=t_{ij}f(\underline{\mu}+\omega(j)),
\end{split}
\end{equation}
where $f,f_1,f_2\in M_{\mathfrak{h}^*}$. The second set of relations are 
\begin{equation}\label{relation}
\begin{split}
h(\mu_k-\mu_l)t_{ik}t_{il}&=t_{il}t_{ik},\quad k<l\\
h(\lambda_j-\lambda_i)t_{jk}t_{ik}&=t_{ik}t_{jk},\quad i<j\\
t_{ik}t_{jl}=t_{jl}t_{ik}&+\left(h(\lambda_j-\lambda_i)-h(\mu_k-\mu_l)\right)t_{jk}t_{il},\quad i<j,k<l\\
g(\mu_k-\mu_l)t_{ik}t_{jl}&=
g(\lambda_i-\lambda_j)t_{jl}t_{ik}\\
&+\left(h(\mu_l-\mu_k)-h(\lambda_i-\lambda_j)\right)t_{il}t_{jk}, i<j,k<l.
\end{split}
\end{equation}

The algebra $\mathscr{F}_R(M(n))$ has a bigradation defined as follows.
Let $deg(t_{ij})=(\omega(i), \omega(j))=(e_i, e_j)\in \mathbb N^n\times \mathbb N^n$ and extend this multiplicatively. Then
$$
\mathscr{F}_R(M(n))=\bigoplus_{(m,p)\in \mathbb N^n\times\mathbb N^n\cup (0, 0)} \mathscr F_{m, p}
$$
where the summand
$f(\underline{\lambda}),f(\underline{\mu})\in \mathscr{F}_{00}=M_{\mathfrak h}^{\otimes 2}$, and $\mathscr F_{m, p}=\{t| deg(t)=(m, p)\in \mathbb N^n\times \mathbb N^n\}$,
and the moment maps are given by
$\mu_{l}(f)=f(\lambda), \mu_{r}(f)=f(\mu)$.

For $\alpha\in \mathfrak h^*$ we denote by $T_{\alpha}: M_{\mathfrak h^*}
\rightarrow M_{\mathfrak h^*}$ the automorphism
$(T_{\alpha}f)(\lambda)=f(\lambda+\alpha)$ for all $\lambda\in \mathfrak h^*$.
The algebra $\mathscr{F}_R(M(n))$ has a comultiplication
$\Delta: \mathscr{F}_R(M(n))\longrightarrow \mathscr{F}_R(M(n))\ot \mathscr{F}_R(M(n))$ given by
\begin{equation}
\begin{split}
\Delta(t_{ij})&=\sum_kt_{ik}\otimes t_{kj},\\
\Delta(f(\underline{\lambda}))&=f(\underline{\lambda})\otimes 1,\\
\Delta(f(\underline{\mu}))&=1\ot f(\underline{\mu}).
\end{split}
\end{equation}
and the counit $\varepsilon$ given by $\varepsilon(t_{ij})=\delta_{ij}T_{-\omega(i)}$,
$\varepsilon(f(\underline{\lambda}))=\varepsilon(f(\underline{\mu}))=f$
and the map is extended as a homomorphism and
we equip $\mathscr{F}_R(M(n))$ with the structure of an $\mathfrak{h}$-bialgebroid.

\begin{definition}
An $\mathfrak{h}$-space $V$ is a vector space over $M_{\mathfrak{h}^*}$ equipped with a diagonalizable $\mathfrak{h}$-module, i.e.
$V=\sum_{\alpha\in \mathfrak{h}^*}V_{\alpha}$, with $M_{\mathfrak{h}^*}V_{\alpha}\in V_{\alpha} $ for all $\alpha\in \mathfrak{h}^*$.
A morphism of $\mathfrak{h}$-spaces is an $\mathfrak{h}$-invariant $\mathfrak{h}^*$-linear map.
\end{definition}

We next define the tensor product of an $\mathfrak{h}$-bialgebroid $\mathcal{A}$ and an $\mathfrak{h}$-space $V$.
Put $V\widetilde{\ot }\mathcal{A}=\bigoplus _{\alpha,\beta\in \mathfrak{h}^*}(V_{\alpha}\ot_{\mathfrak{h}^*} \mathcal{A}_{\alpha\beta} )$
where $\ot_{\mathfrak{h}^*}$ denotes the usual tensor product modulo the relations
$v\ot \mu_l(f)a =fv\ot a$.
The grading $V_{\alpha}\ot_{\mathfrak{h}^*} \mathcal{A}_{\alpha\beta} \subseteq (V\ot \mathcal{A})_\beta$ for all $\alpha$ and $f(v\ot a)=v\ot\mu_r(f)a$ makes $V\widetilde{\ot}\mathcal{A}$ into an $\mathfrak{h}$-space.
Analogously $\mathcal{A} \widetilde{\ot} V =\bigoplus _{\alpha,\beta\in \mathfrak{h}^*}(\mathcal{A}_{\alpha\beta} \ot_{\mathfrak{h}^*}V_{\beta}  )$
where $\ot_{\mathfrak{h}^*}$ denotes the usual tensor product modulo the relations
$ \mu_r(f)a \ot v=a\ot fv$.
The grading  $\mathcal{A}_{\alpha\beta} \ot_{\mathfrak{h}^*}V_{\beta} \subseteq (  \mathcal{A}\ot V)_\alpha$ and $f(a\ot v)=\mu_l(f)a\ot v$ makes
$\mathcal{A}\widetilde{\ot} V$ into an $\mathfrak{h}$-space.

We now construct two special $\mathscr{F}_R(M(n))$-comodules. Let $W=M_{\mathfrak h}\langle w_i\rangle$ be the unital associative algebra generated by the elements $w_i, 1\leq i\leq n$
and $M_{\mathfrak{h}^*}$, its elements denoted by $f(\underline{\lambda})$,
subject to the relations
\begin{equation}\label{relation W}
\begin{split}
w_i^2&=0,  1\leq i\leq n\\
w_jw_i&=-h(\lambda_j-\lambda_i)w_iw_j , 1\leq i<j \leq n,
\end{split}
\end{equation}
as well as the relation $f(\underline{\lambda})w_i=w_if(\underline{\lambda}+\omega(i))$ for all $f\in M_{\mathfrak{h}^*}$.

Let $V=M_{\mathfrak h}\langle v_i\rangle$ be the unital associative algebra generated by the elements $v_i, 1\leq i\leq n$
and a copy of $M_{\mathfrak{h}^*}$ embedded as a subalgebra, its elements denoted by $f(\underline{\lambda})$
subject to the relations
\begin{equation}\label{relation V}
\begin{split}
v_i^2&=0,  1\leq i\leq n\\
v_iv_j&=-h(\lambda_i-\lambda_j)v_jv_i , 1\leq i<j \leq n,
\end{split}
\end{equation}
plus that $f(\underline{\lambda})v_i=v_if(\underline{\lambda}+\omega(i))$ for all $f\in M_{\mathfrak{h}^*}$.

The following result is easy to see.

\begin{theorem}\cite{KN}
Define $\alpha_R(1)=1\ot 1$, $\alpha_R(w_i)=\sum_{j=1}^n w_j\ot t_{ji}$, $\alpha_L(1)=1\ot 1$, $\alpha_L(v_i)=\sum_{j=1}^n t_{ij}\ot v_j$.
Then $\alpha_L$ extends uniquely to
$\alpha_L: V\rightarrow \mathscr{F}_R(M(n))\ot V $
such that $V$ is a left $\mathfrak{h}$-comodule algebra for $\mathscr{F}_R(M(n))$ and $\alpha_R$ extends uniquely to
$\alpha_R: W\rightarrow W\ot \mathscr{F}_R(M(n))$
such that $W$ is a right $\mathfrak{h}$-comodule algebra for $\mathscr{F}_R(M(n))$.
\end{theorem}

Let $I$ be a subset of $[1,n]$ with entries $i_1<i_2<\cdots<i_r$ and
$S_r$ be the symmetric group in $r$ letters. For an element $\sigma\in S_r$, we use $l(\sigma)$ denote the length of $\sigma$. The generalized sign functions $S(\sigma,I)$ and $\tilde{S}(\sigma,I)$ are defined as follows:
\begin{align}\notag
S(\sigma,I)(\underline{\lambda})=\prod_{1\leq k<l\leq r:\sigma(k)>\sigma(l)}(-h(\lambda_{i_{\sigma(k)}}-\lambda_{i_{\sigma(l)}}))\\
=(-q)^{l(\sigma)} \prod_{1\leq k<l\leq r:\sigma(k)>\sigma(l)}
\frac{q^{-2\lambda_{i_{\sigma(k)}}}-q^{-2}q^{-2\lambda_{i_{\sigma(l)}}}}{q^{-2\lambda_{i_{\sigma(k)}}}-q^{-2\lambda_{i_{\sigma(l)}}}}
\end{align}

\begin{align}
\tilde{S}(\sigma,I)(\underline{\lambda})=\prod_{k<l:\sigma(k)>\sigma(l)}(-h(\lambda_{i_{\sigma(l)}}-\lambda_{i_{\sigma(k)}}))
=\frac{1}{S(\sigma,I)(\underline{\lambda}+\underline{1})}.
\end{align}
where $\underline{1}=(1, \ldots, 1)$.

For two subsets $I, J\subset [1, n]$ with $|I|=|J|=r$, the dynamical quantum column minor determinants $\xi^I_J$ and row minor determinants $\eta^I_J$ and  are defined as follows:
\begin{align}
\xi^I_J&=\mu_r(S(\rho,J)^{-1})\sum_{\sigma\in S_r}\mu_{l}(S(\sigma,I))t_{i_{\sigma(1)}j_{\rho(1)}}
\cdots t_{i_{\sigma(r)}j_{\rho(r)}}\\
\eta^I_J&=\mu_l(\tilde{S}(\rho,I)^{-1})\sum_{\sigma\in S_r}\mu_{r}(\tilde{S}(\sigma,J))t_{i_{\rho(r)}j_{\sigma(r)}}
\cdots t_{i_{\rho(1)}j_{\sigma(1)}}
\end{align}
where $\rho\in S_r$.
Using the comodule structures of $\alpha_L$ and $\alpha_R$ the following result can be easily obtained.

\begin{theorem}\cite{KN}
$\xi^I_J=\eta^I_J$ in $\mathscr{F}_R(M(n))$.
\end{theorem}

The element  $\det=\xi^{\{1,2\ldots,n\}}_{\{1,2,\ldots,n\}}$ will be called the (quantum dynamic) determinant of $\mathscr{F}_R(M(n))$.

For ordered subsets $I, J$ one has that
\begin{align}
\alpha_L(v_I)&=\sum_{|K|=|I|}\xi^I_K\ot v_{K},\\
\alpha_R(w_J)&=\sum_{|K|=|J|}w_{K}\ot \xi^K_J,\\
\Delta(\xi^I_J)&=\sum_{|K|=|I|}\xi^I_K \ot \xi^K_J.
\end{align}

For disjoint ordered subsets $I_1,I_2$ of $\{1,2,\ldots ,n\}$, we introduce the quantum sign element $\sign(I_1,I_2)$ inside
$M_{\mathfrak{h}^*}$ by

\begin{equation}
\sign(I_1,I_2)=\prod_{k>l;k\in I_1,l\in I_2}(-h(\lambda_k-\lambda_l)).
\end{equation}

\begin{proposition}\cite{KN}\label{det lap}
Let $I, J_1,J_2$ be subsets of $\{1,2,\ldots,n\}$. If $J=J_1\cup J_2$,
$|I|=|J|$. Then
\begin{align}
\mu_r(\sign(J_1;J_2))\xi^I_J=\sum_{I_1\cup I_2=I}\mu_l(\sign(I_1;I_2))\xi^{I_1}_{J_1}\xi^{I_2}_{J_2},\\
\xi^J_I=\sum_{I_1\cup I_2=I}\xi^{J_1}_{I_1}\frac{\mu_l(\sign(J_2;J_1))}{\mu_r(\sign(I_2;I_1))}\xi^{I_2}_{J_2}.
\end{align}
\end{proposition}

It is easy to see from Proposition \ref{det lap} that
for any $i,j\in [1,n]$
\begin{align}
\begin{split}
\delta_{ij}\det=\sum_{k=1}^{n}\frac{\sign(\{k\};\hat{k})(\underline{\lambda})}{\sign(\{i\};\hat{i})(\underline{\mu})}
t_{kj}\xi^{\hat{k}}_{\hat{i}}, \\
\delta_{ij}\det=\sum_{k=1}^{n}t_{jk}\frac{\sign(\hat{i};\{i\})(\underline{\lambda})}{\sign(\hat{k};\{k\})(\underline{\mu})}
\xi^{\hat{i}}_{\hat{k}},\\
\delta_{ij}\det=\sum_{k=1}^{n}\frac{\sign(\hat{k};\{k\})(\underline{\lambda})}{\sign(\hat{i};\{i\})(\underline{\mu})}
\xi^{\hat{k}}_{\hat{i}}t_{kj}, \\
\delta_{ij}\det=\sum_{k=1}^{n}\xi^{\hat{i}}_{\hat{k}}\frac{\sign(\{i\};\hat{i})(\underline{\lambda})}{\sign(\{k\};\hat{k})(\underline{\mu})}
t_{jk}.
\end{split}
\end{align}

\begin{lemma}\cite{KN}
In $\mathscr{F}_R(M(n))$, the determinant commutes with all quantum minor determinants. In particular,
$\det$ commutes with all generators $t_{ij}$. Moreover, $\Delta(\det)=\det\ot \det$ and $\varepsilon(\det)=T_{-\underline{1}}$, with $\underline{1}=(1,\ldots,1)\in \mathfrak{h}^*$.
\end{lemma}

\begin{proposition}\cite{KN}
The $\mathfrak{h}$-bialgebroid $\mathscr{F}_R(M(n))$ is an $\mathfrak{h}$-Hopf algebroid with the antipode
S defined on the generators by $S(\det^{-1})=\det$, $S(\mu_r(f))=\mu_l(f),S(\mu_l(f))=\mu_r(f)$ for all
$f\in M_{\mathfrak{h}^*}$ and
$$S(t_{ij})={\det}^{-1}\frac{\mu_l(\sign(\hat{j};\{j\}))}{\mu_r(\sign(\hat{i};\{i\}))}
\xi^{\hat{j}}_{\hat{i}}$$
and extended as an algebra antihomomorphism.
\end{proposition}

\section{Quasideterminants and Dieudonn\'e determinants}

Throughout this section we work with rings of fractions of noncommutative rings.

\begin{definition}\cite{GR}
Let $\mathrm X=(x_{ij})$  be an $n\times n$ matrix over a ring with identity such that its inverse matrix $\mathrm X^{-1}$ exists, and the $(j,i)$th entry of $\mathrm X^{-1}$ is an invertible element
of the ring. Then the $(ij)$th quasideterminant of X is defined by the formula
\begin{equation*}
|\mathrm X|_{ij}=
          \left| \begin{array}{ccccc}
             x_{11} & \ldots & x_{1j} & \ldots & x_{1n} \\
              & \ldots &  & \ldots &  \\
             x_{i1} & \ldots &\boxed{x_{ij}} & \ldots & x_{in} \\
              & \ldots &  & \ldots &  \\
             x_{n1} & \ldots & x_{nj} & \ldots & x_{nn} \\
           \end{array}
         \right|=(\mathrm X^{-1})_{ji}^{-1},
\end{equation*}
where the first or the second notation with $\boxed{x_{ij}}$ denotes the quasideterminant.
\end{definition}

When $n\geq 2$, and let $\mathrm X^{ij}$ be the $(n-1)\times(n-1)$-matrix obtained from X by
deleting the $i$th row and $j$th column. In general $\mathrm X^{i_1\cdots i_r, j_1\cdots j_r}$ denotes
the submatrix obtained from $\mathrm X$ by deleting the $i_1, \cdots, i_r$-th rows, and
$i_1, \cdots, i_r$-th columns. Then
$$|\mathrm X|_{ij}=x_{ij}-\sum_{i', j'} x_{ii'}(|\mathrm X^{ij}|_{j'i'})x_{j'j},$$
where the sum runs over $i'\notin I\setminus\{i\}, j'\notin J\setminus\{j\}$.

\begin{theorem}\label{quasidet}
Let $T$ be the matrix of generators $t_{ij}$ of $\mathscr{F}_R(M(n))$, $\sigma=i_1\ldots i_n$ and $\tau=j_1\ldots j_n$ be two permutations of $S_n$.
In the ring of fractions of $\mathscr{F}_R(M(n))$, one has that
 \begin{equation}\label{det quasidet2}
 \det(T)=\frac{\prod_{k=1}^{n}\mu_l(\sign({I_k}^c;i_k))}{\prod_{k=1}^{n}\mu_r(\sign({J_k}^c;j_k))}
t_{i_nj_n} \ldots |T^{i_1j_1}|_{i_2j_2} |T|_{i_1j_1}
 \end{equation}
where $I_k=\{i_1,\ldots,i_k\}$, $J_k=\{j_1,\ldots,j_k\}$.
\end{theorem}

\begin{proof}
By definition the quasi-determinants of $T$ are  inverses of the entries of $S(T)$,
\begin{equation}
\begin{aligned}
|T|_{ij}=S(t_{ji})^{-1}&
={\xi^{\hat{i}}_{\hat{j}}}^{-1} \frac{\mu_r(\sign(\hat{j};\{j\}))}{\mu_l(\sign(\hat{i};\{i\}))}
\det(T),
\end{aligned}
\end{equation}
then
\begin{equation}
\begin{aligned}
\det(T)=\frac{\mu_l(\sign(\hat{i};\{i\}))}{\mu_r(\sign(\hat{j};\{j\}))} {\xi^{\hat{i}}_{\hat{j}}}|T|_{ij}.
\end{aligned}
\end{equation}
Eqs (\ref{det quasidet2}) follows from induction on $n$.
\end{proof}
\begin{remark}
If $i_k=j_k=n+1-k$ for any $k$, all the factors on the right hand side of \ref{det quasidet2} commute with each other.
In general the factors do not commute.
\end{remark}

\section{Dynamical quantum Pfaffians}

First we review the general theory of the Pfaffian \cite{JZ1, JZ2}, and we assume
the minimum condition here. Let $\mathcal{B}$ be the algebra generated by the elements
$b_{ij}$ for $1\leq i<j\leq 2n$,
and a copy of $M_{\mathfrak{h}^*}$ embedded as a subalgebra, its elements denoted by $f(\underline{\lambda})$.
The dynamical quantum Pfaffian is defined by
\begin{align*}
\Pf (B)
=\sum_{\sigma\in \Pi}S(\sigma)b_{\sigma(1)\sigma(2)}b_{\sigma(3)\sigma(4)}\cdots b_{\sigma(2n-1)\sigma(2n)},
\end{align*}
where $S(\sigma)=S(\sigma,[1,2n])$, $\Pi$ is the set of permutations $\sigma$ of $2n$ such that
$\sigma(2i-1)<\sigma(2i), i=1,\ldots,n.$

For any two disjoint subsets $I_1, I_2$ of $[1,2n]$, we define the dynamical quantum sign functions $\sign(I_1,I_2)$ and $\widetilde{\sign}(I_1,I_2)$ by
\begin{equation}
\begin{split}
\sign(I_1,I_2)=\prod_{k>l;k\in I_1,l\in I_2}(-h(\lambda_k-\lambda_l)),\\
\widetilde{\sign}(I_1,I_2)=\prod_{k<l;k\in I_1,l\in I_2}(-h(\lambda_k-\lambda_l)).
\end{split}
\end{equation}

Let $I=\{i_1,i_2,\dots,i_k\}$ with $1\leq i_1<i_2<\dots,i_k\leq 2n$. Denote by $B_I$ the submatrix of $B$ with the rows and  columns indexed by $I$.
The following result gives an iterative algorithm to compute the dynamic Pfaffian.

\begin{proposition}\label{Pfaffian-Lap}
For each $0\leq t\leq n$ we have that
\begin{equation}
\Pf(B)=\sum_{I}\sign(I;I^c)\Pf(B_{I})\Pf(B_{I^c}),
\end{equation}
where the sum runs over all subsets $I$  of $[1,2n]$ such that $|I|$=2t
and $I^c$ is the complement of $I$.
\end{proposition}
\begin{proof}
Define the tensor product of $W\ot \mathcal{B}$ to be the usual tensor product
modulo the relations
$fw\ot b=w\ot f b$.
Let $\Omega=\sum_{i<j}w_iw_j\ot b_{ij}$, then
\begin{equation}
\bigwedge^{n}\Omega=w_{1}\wedge\cdots \wedge w_{2n}\ot \Pf (B).
\end{equation}

On the other hand,

\begin{align}\notag
\bigwedge^{n}\Omega&=\Omega^t\bigwedge\Omega^{n-t}\\ \notag
&=\sum_{I,J}\left(w_I\ot \Pf(B_{I})\right) \left(w_{J}\ot \Pf(B_{J})\right)\\
&=\sum_{I,J}w_Iw_{J}\ot \Pf(B_{I})\Pf(B_{J}).
\end{align}
Note that $w_Iw_{J}$ vanishes unless $J=I^c$, therefore we have that
$$
\Pf(B)=\sum_{I}\sign(I;I^c) \Pf(B_{I})\Pf(B_{I^c}).
$$
\end{proof}

The following transformation formula establishes the relation between the quantum dynamic Pfaffian and determinant.
\begin{theorem}\label{pf eq1}
Denote by $\mathscr{F}_R(M(2n))\widetilde{\ot} \mathcal{B}$ the usual tensor product modulo the relations
$\mu_r(f)t\ot b=t\ot f b$ and $f(t\ot b)=\mu_l(f)t\ot b $.
Let $\xi^{ij}_{kl}$ be the $2\times 2$-dynamical quantum minors in $\mathscr{F}_R(M(2n))$, $b_{ij}$ the generators of $\mathcal{B}$, and $c_{ij}=\sum_{k<l}\xi^{ij}_{kl}\ot b_{kl}$.
Then in $\mathscr{F}_R(M(2n))\widetilde{\ot} \mathcal{B}$ we have
$$\Pf(C)={\det}(T)\ot \Pf(B).$$
\end{theorem}

\begin{proof}
 Let $w\ot t\ot b $ be an element of $W\widetilde{\ot }\mathscr{F}_R(M(n))\widetilde{\ot}\mathcal{B}$,
then
\begin{equation}
\begin{split}
fw\ot t\ot b= w\ot \mu_l(f)t\ot b=w\ot f(t\ot b),\\
f(w\ot t)\ot b=w\ot \mu_r(f)t\ot b=w\ot t\ot f(b).
\end{split}
\end{equation}

Let $\delta_i=\sum_{j=1}^{2n}w_j\ot t_{ji}$, and consider the element
$\Omega=\sum_{i, j} w_iw_j\ot c_{ij}$. It is clear that
\begin{equation}\label{hyperpf1}\Omega^n=w_1\cdots w_{2n}\ot \Pf(C),
\end{equation}
where the product is the wedge product among $w_i$.
On the other hand,
$\Omega=\sum \delta_{i}\delta_{j}\ot b_{ij}$. Then
\begin{equation}\label{hyperpf2}
\Omega^n=\delta_1 \cdots  \delta_{2n} \ot \Pf(B)=w_1 \cdots w_{2n}\ot {\det}(T)\ot \Pf(B).
\end{equation}

Comparing (\ref{hyperpf1}) and (\ref{hyperpf2}) we conclude that
$$\Pf(C)={\det}(T)\ot \Pf(B).$$
\end{proof}

Let $\widetilde {\mathcal{B}}$ be the algebra generated by the elements
$\widetilde b_{ji}$ for $1\leq i<j\leq 2n$,
and a copy of $M_{\mathfrak{h}^*}$ embedded as a subalgebra, its elements denoted by $f(\underline{\lambda})$.
the dynamical quantum Pfaffian is defined by
\begin{align*}
\widetilde{\Pf} (\widetilde B)
=\sum_{\sigma\in \Pi}\tilde{S}(\sigma)\widetilde b_{\sigma(2n)\sigma(2n-1)}\widetilde b_{\sigma(2n-2)\sigma(2n-3)}\cdots \widetilde b_{\sigma(2)\sigma(1)},
\end{align*}
where $\Pi$ is the set of shuffle permutations $\sigma$ of $2n$ such that
$\sigma(2i-1)<\sigma(2i), i=1,\ldots,n.$

For any two disjoint subsets $I_1, I_2$ of $\{1, 2, \ldots, 2n\}$, we define the dynamical quantum sign by
\begin{equation}
\widetilde{\sign}(I_1,I_2)=\prod_{k<l;k\in I_1,l\in I_2}(-h(\lambda_k-\lambda_l)).
\end{equation}

Similarly, we have the following statements.
\begin{proposition}\label{Pf-Lap}
For any $0\leq t\leq n$, we have that
\begin{equation}
\widetilde{\Pf}(\widetilde B)=\sum_{I}\widetilde{\sign}(I;I^c)\widetilde{\Pf}(\widetilde B_{I})\widetilde{\Pf}(\widetilde B_{I^c}),
\end{equation}
where the sum runs through all subsets $I$  of $[1,2n]$ such that $|I|$=2t.
\end{proposition}
\begin{proof}
Define the tensor product of $\widetilde {\mathcal{B}}\ot V$ to be the usual tensor product
modulo the relations
$f b\ot v=b\ot f v$.
Let $\widetilde {\Omega}=\sum_{i>j}b_{ij}\ot v_iv_j$, then
\begin{equation}
\bigwedge^{n} \widetilde{\Omega}=\widetilde{\Pf} (\widetilde B)\otimes v_n \wedge\cdots \wedge v_1.
\end{equation}

On the other hand,

\begin{align}\notag
\bigwedge^{n} \widetilde {\Omega} &=\widetilde{\Omega}^t\bigwedge\widetilde{\Omega}^{n-t}\\ \notag
&=\sum_{I,J}\left(\widetilde{\Pf}(\widetilde{B}_{I})\ot v_I\right) \left(\widetilde{\Pf}(\widetilde B_{J})\ot v_J\right)\\
&=\sum_{I,J}\widetilde{\Pf}(\widetilde B_{I})\widetilde{\Pf}(\widetilde B_{J})\ot v_Iv_{J}.
\end{align}
It is easy to see that $v_Iv_{J}$ vanishes unless $J=I^c$.
Thus we conclude that
$$
\widetilde{\Pf}(\widetilde B)=\sum_{I}\widetilde{\sign}(I;I^c)\Pf(\widetilde B_{I})\Pf(\widetilde B_{I^c}).
$$
\end{proof}

\begin{theorem}\label{pf eq2}
Denotes by $\widetilde{\mathcal{B}} \widetilde{\ot} \mathscr{F}_R(M(2n))$ the usual tensor product modulo the relations
$b\ot \mu_l(f)t=f b\ot t$ and $f(b\ot t)=b\ot \mu_r(f)t $.
Let $\xi^{kl}_{ij}$ be the dynamical quantum minor in $\mathscr{F}_R(M(n))$, 
$c_{ji}=\sum_{k<l}\widetilde b_{lk}\ot \xi^{kl}_{ij}$.
Then in $\widetilde{\mathcal{B}} \widetilde{\ot} \mathscr{F}_R(M(2n))$ we have $\widetilde{\Pf}(C)=\widetilde{\Pf}(\widetilde B)\ot {\det}(T)$.
\end{theorem}

\begin{proof}
Let $b\ot t\ot v $ be an element of $\mathcal{B}\widetilde{\ot }\mathscr{F}_R(M(n))\widetilde{\ot}V$,
then
\begin{equation}
\begin{split}
b\ot t\ot f v=b \ot \mu_r(f)t\ot v=f(b \ot t)\ot v,\\
b\ot f(t\ot v) =b\ot \mu_l(f)t\ot v=f(b)\ot t\ot v.
\end{split}
\end{equation}

Let $\delta_i=\sum_{j=1}^{2n}t_{ij}\ot v_{j}$, and consider the element
$\widetilde {\Omega}=\sum  c_{ji}\ot v_j v_i$. It is clear that
\begin{equation}\label{hyperpf1}
\widetilde{\Omega}^n=\widetilde{\Pf}(C)\ot v_{2n}\cdots v_{1}.
\end{equation}

On the other hand,
$\widetilde{\Omega}=\sum b_{lk}\ot \delta_{l}\delta_{k}$. Then
\begin{equation}\label{hyperpf2}
\widetilde{\Omega}^n=\widetilde{\Pf}(C)\ot  \delta_{2n} \cdots  \delta_{1}=\widetilde{\Pf}(\widetilde B)\ot {\det}(T)\ot  \delta_{2n} \cdots  \delta_{1}.
\end{equation}

Comparing (\ref{hyperpf1}) and (\ref{hyperpf2}) we conclude that
$$\widetilde{\Pf}(C)=\widetilde{\Pf}(\widetilde B)\ot {\det}(T).$$
\end{proof}

We now generalize the notion of the dynamical quantum Pfaffian to the dynamical quantum hyper-Pfaffian.
Let $\mathcal{B}$ be the algebra generated by the elements
 $b_{i_1\cdots i_{m}},i_1<i_2<\cdots<i_m, 1\leq i_k\leq mn,k=1,\ldots, m$,
and a copy of $M_{\mathfrak{h}^*}$ embedded as a subalgebra, its elements denoted by $f(\underline{\lambda})$.
The dynamical quantum hyper-Pfaffian is defined by

\begin{align*}
{\Pf}_m(B)
=\sum_{\sigma\in \Pi}{S}(\sigma)b_{\sigma(1)\cdots\sigma(m)}\cdots b_{\sigma(m(n-1)+1)\cdots\sigma(mn)},
\end{align*}
where $\Pi$ is the set of permutations $\sigma$ of $mn$ such that
$\sigma((k-1)m+1)<\sigma((k-1)m+2)<\cdots <\sigma(km), k=1,\ldots,n.$

Note that the dynamical hyper-Pfaffian uses only the entries $b_{i_1\cdots i_m}$, where
$i_1<\cdots <i_m$. Clearly ${\Pf}_2(B)={\Pf}(B)$, the quantum dynamical Pfaffian.

\begin{proposition}\label{hyPfaffian-Lap}
For any $0\leq t\leq n$,
\begin{equation}
{\Pf}_m(B)=\sum_{I}\sign(I;I^c){\Pf}_m(B_{I}){\Pf}_m(B_{I^c}),
\end{equation}
where the sum is taken over all subset $I$  of $[1,mn]$ such that $|I|=mt$.
\end{proposition}
%
%

Let $I=\{i_1,i_2,\cdots,i_m\}$ such that $i_1<i_2<\cdots <i_m$. We denote by $b_I$ the element $b_{i_1\cdots i_m}$.
\begin{theorem}\label{hypf eq1}
Denotes by $\mathscr{F}_R(M(mn))\widetilde{\ot} \mathcal{B}$ the usual tensor product modulo the relations
$\mu_r(f)t\ot b=t\ot f b$ and $f(t\ot b)=\mu_l(f)t\ot b $.
Let $\xi^{I}_{J}$ be the dynamical quantum minor in $\mathscr{F}_R(M(n))$, $c_{I}=\sum_{J}\xi^{I}_{J}\ot b_{J}.$
Then in $\mathscr{F}_R(M(mn))\widetilde{\ot} \mathcal{B}$ we have
${\Pf}_m(C)={\det}(T)\ot {\Pf}_m(B)$.
\end{theorem}

Let $\widetilde{\mathcal{B}}$ be the algebra generated by the elements
$\widetilde b_{i_1\cdots i_{m}},i_1>i_2>\cdots>i_m, 1\leq i_k\leq mn,k=1,\ldots, m.$
Define dynamical quantum hyper-Pfaffian by

\begin{align*}
\widetilde{\Pf}_m (\widetilde B)
=\sum_{\sigma\in \Pi}\widetilde{S}(\sigma)\widetilde b_{\sigma(mn)\cdots\sigma(mn-m+1)} \widetilde b_{\sigma(mn-m)\sigma(mn-2m+1)}\cdots \widetilde b_{\sigma(m)\cdots\sigma(1)},
\end{align*}
where $\Pi$ is the set of permutations $\sigma$ of $mn$ such that
$\sigma((k-1)m+1)<\sigma((k-1)m+2)<\cdots <\sigma(km), k=1,\ldots,n.$

Clearly $\widetilde{\Pf}_2 (\widetilde B)=\widetilde{\Pf}(\widetilde B)$ discussed before.

\begin{proposition}\label{Pf-Lap}
For any $0\leq t\leq n$,
\begin{equation}
\widetilde{\Pf}_m(\widetilde B)=\sum_{I}\widetilde{\sign}(I;I^c)\widetilde{\Pf}_m(\widetilde B_{I})\widetilde{\Pf}_m(\widetilde B_{I^c}),
\end{equation}
where the sum is taken over all subset $I$  of $[1,mn]$ such that $|I|=mt$.
\end{proposition}

Let $I=\{i_1,i_2,\cdots,i_m\}$ such that $i_1<i_2<\cdots <i_m$. We denote by $ \widetilde b_I$ the element $b_{i_m\cdots i_1}$.
The following result is proved by the similar method as in the case of $m=2$.

\begin{theorem}

Denotes by $\widetilde{\mathcal{B}} \widetilde{\ot} \mathscr{F}_R(M(mn))$ the usual tensor product modulo the relations
$b\ot \mu_l(f)t=f b\ot t$ and $f(b\ot t)=b\ot \mu_r(f)t $.
Let $\xi^{J}_{I}$ be the dynamical quantum minor in $\mathscr{F}_R(M(n))$, $c_{I}=\sum_{J} \widetilde b_{J}\ot \xi^{J}_{I}$.
Then in $\widetilde{\mathcal{B}} \widetilde{\ot} \mathscr{F}_R(M(mn))$ we have $\widetilde{\Pf}_m(C)=\widetilde{\Pf}_m( \widetilde B)\ot {\det}(T)$.
\end{theorem}

\bigskip
\centerline{\bf Acknowledgments}
\medskip
The work is supported by National Natural Science Foundation of China (grant no. 11531004), Fapesp (grant no. 2015/05927-0)
and Humboldt foundation.
Jing acknowledges the support of
Max-Planck Institute for Mathematics in the Sciences, Leipzig during the research.

\vskip 0.1in

\bibliographystyle{amsalpha}

\end{document}